\documentclass[11pt]{article}

\usepackage[margin=1.08in]{geometry}
\usepackage{amsmath,amssymb,amsthm}
\usepackage{xcolor}
\usepackage[
  colorlinks=true,
  linkcolor=blue!60!black,
  citecolor=blue!60!black,
  urlcolor=blue!60!black
]{hyperref}
\hypersetup{
  pdftitle={Characteristic drops for high-order vanishing on the hypercube},
  pdfauthor={Dean Menezes}
}

\newtheorem{theorem}{Theorem}[section]
\newtheorem{proposition}[theorem]{Proposition}
\newtheorem{lemma}[theorem]{Lemma}
\newtheorem{corollary}[theorem]{Corollary}
\theoremstyle{definition}
\newtheorem{definition}[theorem]{Definition}

\DeclareMathOperator{\ord}{ord}
\newcommand{\Z}{\mathbb Z}
\newcommand{\zero}{\mathbf 0}
\newcommand{\cube}{\{0,1\}^n}
\newcommand{\pcube}{\cube\setminus\{\zero\}}
\newcommand{\ghat}{\widehat g}

\title{Characteristic drops for high-order vanishing on the hypercube}
\author{Dean Menezes\thanks{University of Texas at Austin,
Austin, TX 78712, USA\@.  E-mail: \texttt{dean.menezes@utexas.edu}.}}
\date{}

\begin{document}
\maketitle

\begin{abstract}
Let $F$ be a field, let $k\ge2$, let $0\le \ell\le k-1$, and suppose that
$n\ge k-1$.  We determine the minimum degree of a polynomial in
$F[x_1,\ldots,x_n]$ that vanishes to order at least $k$ at every nonzero
vertex of the Boolean cube and to order exactly $\ell$ at the origin.  The
answer is
\[
 n+2k-2-\rho_F(k-\ell),
\]
where $\rho_F(s)$ is the least $r$ for which
$s-1=a_1+\cdots+a_r$ and every $C_{a_i-1}$ is nonzero in $F$; the empty
sum is allowed when $s=1$.

Our construction is integral: it gives a $\mathbb Z$-basis of the reduced
vanishing lattice and an integral diagonal presentation of its top-degree
map, with Catalan numbers on the diagonal.  This presentation persists under
arbitrary base change.  Block coordinates refine the diagonalization and make
the basis compatible with polynomial degree.  In odd characteristic every
drop with $\ell\le k-2$ is either zero or one.  In
characteristic $2$ the answer is
\[
 n+2k-2-s_2(k-\ell-1),
\]
where $s_2$ is binary digit sum.  Thus all characteristic drops in the stable
range are determined exactly.
\end{abstract}



\section{Introduction}\label{sec:intro}

For a polynomial $P$, write $[P]_j$ for its homogeneous part of total
degree $j$.  If $P\in F[x_1,\ldots,x_n]$ and $a\in F^n$, let $\ord_aP$ be
the least $j$ for which $[P(x+a)]_j\ne0$, and set $\ord_a0=\infty$.  We also
write
\[
 \operatorname{in}_aP=[P(x+a)]_{\ord_aP}
\]
when $P\ne0$.  Thus $P$ vanishes to order at least $k$ at $a$ precisely when
$\ord_aP\ge k$.
For $0\le\ell\le k-1$, define
\[
 \delta_F(n;k,\ell)=
 \min\bigl\{\deg P:\ord_aP\ge k\ (a\in\pcube),\ 
                       \ord_{\zero}P=\ell\bigr\}.
\]

Sauermann and Wigderson proved over $\mathbb R$ that
\[
 \delta_{\mathbb R}(n;k,\ell)=n+2k-3
 \qquad(0\le\ell\le k-2,\, n\ge2k-3),
\]
and that $\delta_{\mathbb R}(n;k,k-1)=n+2k-2$ \cite{SW}.
Huang, Wang, and Wang subsequently improved the first range to $n\ge k-1$
as the Boolean case of a theorem for $\{0,1,\ldots,m\}^n$
\cite[Theorem~1.7]{HuangWangWang}.  The proof in \cite{SW} features the
Catalan numbers
\[
 C_j=\frac{1}{j+1}\binom{2j}{j}.
\]
Its distinguished top coordinate is $C_{k-2}$, so the top-degree argument
can fail in positive characteristic.

To describe all characteristic drops, let $s\ge1$ and define
\begin{equation}\label{eq:rho-def}
 \rho_F(s)=\min\left\{r\ge0:
 \begin{array}{c}
  s-1=a_1+\cdots+a_r,\quad a_i\ge1,\\[2pt]
  C_{a_i-1}\ne0\text{ in }F\quad(1\le i\le r)
 \end{array}\right\}.
\end{equation}
Here the empty sum is allowed, so $\rho_F(1)=0$.  For $s\ge2$ the set is
nonempty because $C_0=1$, and hence $1\le\rho_F(s)\le s-1$.
We call $\rho_F(s)$ the \emph{Catalan length} of $s-1$ over $F$.
In characteristic $p$ we also write $\rho_p$.  In characteristic $0$,
$\rho_F(s)=1$ for $s\ge2$.

Our main result is the following exact formula.

\begin{theorem}\label{thm:extremal}
Let $F$ be a field, let $k\ge2$, let $n\ge k-1$, and let
$0\le\ell\le k-1$.  Then
\begin{equation}\label{eq:main-formula}
 \boxed{\ \delta_F(n;k,\ell)=n+2k-2-\rho_F(k-\ell).\ }
\end{equation}
For the boundary value $\ell=k-1$, the formula holds already for $n\ge1$.
\end{theorem}

The structural result behind Theorem~\ref{thm:extremal} is more precise.
Sauermann and Wigderson define a finite-dimensional space $V_k$ of reduced
polynomials that vanish to order $k$ on $\pcube$, and a map $\varphi_k$ that
retains the homogeneous part of degree $n+2k-3$.  We construct polynomials
\[
 G_{k,d}=\left(\prod_{i=1}^n(x_i(x_i-1))^{d_i}\right)\ghat_{k-|d|},
 \qquad d\in\Z_{\ge0}^n,\quad |d|\le k-2,
\]
where $\ghat_s$ is an explicit integral Catalan polynomial.  They form a
$\Z$-basis of the integral reduced vanishing lattice.  If $n\ge k-1$, then
\begin{equation}\label{eq:intro-diagonal}
 \varphi_k(G_{k,d})=C_{k-|d|-2}\,w_{k,d}
\end{equation}
for the standard basis $\{w_{k,d}\}$ of the top-degree space $W_k$.
Consequently the integral top map has the diagonal presentation
\begin{equation}\label{eq:intro-cokernel}
 \operatorname{coker}\varphi_k^{\Z}\cong
 \bigoplus_{t=0}^{k-2}
 \left(\Z/C_{k-t-2}\Z\right)^{\binom{n+t-1}{n-1}}.
\end{equation}
After base change to $F$, this presentation gives, in particular,
\begin{equation}\label{eq:intro-kernel}
 \ker\varphi_k=
 \operatorname{span}\{G_{k,d}:C_{k-|d|-2}=0\text{ in }F\}.
\end{equation}

The diagonal presentation detects only the top layer.  A block-echelon
argument controls every layer: for each nonzero expansion
$P=\sum_dc_dG_{k,d}$,
\begin{equation}\label{eq:intro-degree}
 \deg P=
 \max_{c_d\ne0}\bigl(n+2k-2-\rho_F(k-|d|)\bigr).
\end{equation}
Moreover, $\operatorname{in}_{\zero}G_{k,d}$ is a nonzero scalar multiple of
$x^d$.  Hence $\ord_{\zero}P$ is the least $|d|$ with $c_d\ne0$, and
\eqref{eq:main-formula} follows.

The Catalan length can be evaluated explicitly.  If $F$ has odd characteristic $p$,
then every positive integer is a sum of at most two Catalan-admissible parts.
Consequently:

\begin{corollary}[Odd characteristic]\label{cor:odd-intro}
Let $F$ have odd characteristic $p$, and assume the hypotheses of
Theorem~\ref{thm:extremal}.  The boundary value is
\[
 \delta_F(n;k,k-1)=n+2k-2.
\]
For $0\le\ell\le k-2$,
\[
 \delta_F(n;k,\ell)=
 \begin{cases}
  n+2k-3,&p\nmid C_{k-\ell-2},\\
  n+2k-4,&p\mid C_{k-\ell-2}.
 \end{cases}
\]
Thus the minimum either stays at its characteristic-zero value or drops by
exactly one degree.
\end{corollary}

In characteristic $2$, the admissible parts are the powers of $2$.

\begin{corollary}[Characteristic two]\label{cor:two-intro}
Let $F$ have characteristic $2$, and assume the hypotheses of
Theorem~\ref{thm:extremal}.  Then
\[
 \delta_F(n;k,\ell)=n+2k-2-s_2(k-\ell-1),
\]
where $s_2(m)$ is the number of $1$'s in the binary expansion of $m$.
\end{corollary}

For $\ell=0$, these statements settle the original punctured-cube problem
in every characteristic and at every multiplicity.  For a prime $p$, write
\[
 k_0(p)=\min\{k\ge2:p\mid C_{k-2}\}.
\]
The first failures are the values highlighted in \cite[\S4.2]{SW}:
\[
 k_0(2)=4,\qquad k_0(3)=7,\qquad
 k_0(p)=\frac{p+5}{2}\quad(p\ge5).
\]
More generally, if $k=k_0(p)+\ell$, then
\[
 \delta_F(n;k,\ell)=n+2k-4.
\]

\paragraph{Relation to earlier work.}
The order-one case is the Alon--F\"uredi theorem \cite{AlonFuredi}.
Related polynomial and covering results include Ball and Serra's punctured
Nullstellensatz \cite{BallSerra}, the almost-$k$-cover theorem of Clifton
and Huang \cite{CliftonHuang}, and the layer-cover theorem of Ghosh, Kayal,
and Nandi \cite{GhoshKayalNandi}.  Multiplicity subspace covers over
$\mathbb F_2$ were studied by Bishnoi, Boyadzhiyska, Das, and M\'esz\'aros
\cite{BishnoiEtAl}, and multiplicity covers of planar grids by Bishnoi,
Boyadzhiyska, Das, and den Bakker \cite{CoveringGrids}.  Ghosh, Kayal, and
Nandi later treated larger exceptional subsets of the cube
\cite{AlmostCovering}.

We use the reduction, dimension count, and top-degree space introduced by
Sauermann and Wigderson \cite{SW}.  Their construction supplies one vector
whose distinguished top coordinate is a Catalan number.  The Catalan truncation
below supplies all vectors $G_{k,d}$ at once, gives an integral basis and a
diagonal presentation, and controls every lower leading layer.  This turns
the first-failure calculation into the general formula
\eqref{eq:main-formula}.  We are unaware of an earlier integral presentation
or degree-compatible basis for this reduced vanishing space.

The range $n\ge k-1$ in characteristic $0$ is also the Boolean case of
\cite[Theorem~1.7]{HuangWangWang}, which treats
$\{0,1,\ldots,m\}^n$ over $\mathbb R$ and proves that the analogous top map
is injective.  That argument does not track modular rank loss; the integral
diagonal presentation here identifies it exactly.  Other nearby
work concerns positive-characteristic Zariski closure of symmetric subsets
\cite{SrinivasanVenkitesh} and higher-multiplicity polynomial and
hyperplane covers of symmetric subsets \cite{GKNV}.  These results do not
contain the arbitrary-characteristic punctured-cube formula proved here.

Carde \cite{Carde} and Ekhad--Zeilberger \cite{EkhadZeilberger} gave short
proofs of Catalan identities arising in \cite{SW}.  The divisibility of
Catalan numbers by primes and prime powers was determined by Alter and
Kubota \cite{AlterKubota}; see also Deutsch--Sagan \cite{DeutschSagan}
for congruence results and the binary valuation used below.

Section~\ref{sec:framework} records the characteristic-free linear-algebra
framework.  Section~\ref{sec:catalan} constructs $\ghat_s$.
Section~\ref{sec:basis} gives the explicit basis and diagonalization, while
Section~\ref{sec:degree} proves the block-echelon degree formula.
Theorem~\ref{thm:extremal} is proved in Section~\ref{sec:extremal}; the
arithmetic consequences follow in Section~\ref{sec:arithmetic}.

\section{The reduced vanishing space}\label{sec:framework}

Throughout this section, $k\ge2$.  Put
\[
 y_i=x_i(x_i-1)\qquad(1\le i\le n).
\]
For a multi-index $e$, write $|e|=e_1+\cdots+e_n$.

\begin{definition}
A polynomial $P\in F[x_1,\ldots,x_n]$ is \emph{$k$-reduced} if
\[
 \deg P\le n+2k-3
 \quad\text{and}\quad
 \sum_{i=1}^n\left\lfloor\frac{e_i}{2}\right\rfloor\le k-1
\]
for every monomial $x^e$ of $P$.  Let $U_k$ be the space of $k$-reduced
polynomials, and let
\[
 V_k=\{P\in U_k:\ord_aP\ge k\text{ for every }a\in\pcube\}.
\]
Let $U_k^{\Z}$ be the group of integral $k$-reduced polynomials, and let
\[
 V_k^{\Z}=\{P\in U_k^{\Z}:\ord_aP\ge k\text{ for every }a\in\pcube\}.
\]
\end{definition}

The following is the reduction in \cite[Claim~2.1]{SW}, with its stronger
low-degree conclusion included.

\begin{lemma}[Reduction]\label{lem:reduction}
For every $Q\in F[x_1,\ldots,x_n]$, there is a polynomial $P\in U_k$ such
that $\deg P\le\deg Q$, $\ord_a(Q-P)\ge k$ for every $a\in\pcube$, and
$\ord_{\zero}(Q-P)\ge k-1$.  If $\deg Q\le n+2k-3$, then $P$ may be chosen
so that $\ord_a(Q-P)\ge k$ for every $a\in\cube$.  If $Q$ has integral
coefficients, so may $P$.
\end{lemma}

\begin{proof}
For a monomial $x^e$, put
\[
 q(e)=\sum_{i=1}^n\left\lfloor\frac{e_i}{2}\right\rfloor.
\]
Call $x^e$ bad of the first kind if $q(e)\ge k$.  Such a monomial is the
unique top-degree monomial of
\[
 y_{i_1}\cdots y_{i_k}x^m,
\]
for suitable, possibly repeated indices $i_1,\ldots,i_k$, and a monomial
$x^m$.  Subtracting the appropriate coefficient multiple cancels $x^e$ and
introduces only monomials of smaller degree.  The subtracted polynomial has
order at least $k$ at every cube point.

A monomial that is not bad of the first kind has $q(e)\le k-1$.  It can
still violate the degree bound only if
\[
 |e|=n+2k-2,\qquad q(e)=k-1,
\]
and every $e_i$ is odd.  It is then the unique top-degree monomial of
\[
 y_{i_1}\cdots y_{i_{k-1}}\prod_{j=1}^n(x_j-1).
\]
Call such a monomial bad of the second kind.  The displayed polynomial has
order at least $k$ on $\pcube$: the $k-1$ factors $y_{i_\nu}$ supply order
$k-1$, and at each nonzero cube point at least one factor $x_j-1$ vanishes.
At the origin it has order $k-1$.

Remove a bad monomial of largest degree at each step.  Each subtraction
reduces either the largest bad degree or the number of bad monomials of that
degree; hence the process terminates.  The output has no monomial of either
kind.  Thus $q(e)\le k-1$ throughout its support, and its degree is at most
$n+2k-3$: a larger degree would force a monomial of the second kind.  The
output is therefore $k$-reduced, and degree never increases.  If
$\deg Q\le n+2k-3$, no monomial of the second kind can occur, so every
subtraction vanishes to order at least $k$ at the origin as well.
Every reducing polynomial used above has integral coefficients, and its
unique top-degree monomial has coefficient $1$.  Thus an integral input
produces integral
multipliers and an integral output.
\end{proof}

Let
\[
 M_t(n)=\#\{e\in\Z_{\ge0}^n:|e|<t\}=\binom{n+t-1}{n}.
\]
The monomial count in \cite[Claim~2.2]{SW} gives
\begin{equation}\label{eq:dim-U}
 \dim U_k=(2^n-1)M_k(n)+M_{k-1}(n).
\end{equation}
We next verify that the corresponding interpolation statement is
characteristic-free.  With $z=(z_1,\ldots,z_n)$, let
\[
 D^\beta P(a)=[z^\beta]P(a+z)
\]
denote the Hasse derivative.

\begin{lemma}[Hasse interpolation]\label{lem:interpolation}
The map that records all $D^\beta P(a)$ with $a\in\pcube$, $|\beta|<k$,
and all $D^\beta P(\zero)$ with $|\beta|<k-1$, is an isomorphism
\[
 U_k\longrightarrow F^{(2^n-1)M_k(n)+M_{k-1}(n)}.
\]
The analogous map from $U_k^{\Z}$ to
$\Z^{(2^n-1)M_k(n)+M_{k-1}(n)}$ is also an isomorphism.  Consequently,
\begin{equation}\label{eq:dim-V}
 \dim V_k=M_{k-1}(n),
\end{equation}
and restriction to the origin data gives an isomorphism
\begin{equation}\label{eq:integral-origin-data}
 V_k^{\Z}\longrightarrow \Z^{M_{k-1}(n)},
 \qquad P\longmapsto\bigl(D^\beta P(\zero)\bigr)_{|\beta|<k-1}.
\end{equation}
\end{lemma}

\begin{proof}
The dimensions agree by \eqref{eq:dim-U}.  To isolate the coordinate
$(a,\beta)$, where $|\beta|<k$ (and $|\beta|<k-1$ if $a=\zero$), take
\[
 Q_{a,\beta}(x)=(x-a)^\beta
   \prod_{i=1}^n\bigl(1-(x_i-a_i)^{2k}\bigr)^k.
\]
At $a$, this is $(x-a)^\beta$ plus terms of order at least $|\beta|+2k$.
At another cube point $b$, choose $i$ with $b_i\ne a_i$.  The polynomial
$1-(x_i-a_i)^{2k}$ vanishes at $b$, so its $k$th power has order at least
$k$.  Thus the recorded data of $Q_{a,\beta}$ are the desired
standard basis vector.  Lemma~\ref{lem:reduction} replaces
$Q_{a,\beta}$ by a polynomial in $U_k$ without changing those data.
Surjectivity, and hence bijectivity, follows.

The same isolating polynomials and the integral reduction prove surjectivity
over $\Z$.  The source and target are free abelian groups of the same rank,
so that map too is an isomorphism.  Setting the punctured-cube coordinates
equal to zero gives \eqref{eq:dim-V} and
\eqref{eq:integral-origin-data}.
\end{proof}

For $d\in\Z_{\ge0}^n$ with $|d|\le k-2$, put
\begin{equation}\label{eq:wkd}
 w_{k,d}=x_1\cdots x_n\,x^{2d}
 \left(x_1^{2(k-|d|)-3}+\cdots+x_n^{2(k-|d|)-3}\right).
\end{equation}
These are homogeneous of degree $n+2k-3$.  Let $W_k$ be their span.
Sauermann and Wigderson proved the following statement over $\mathbb R$
\cite[Claim~2.7]{SW}.  We include their short monomial argument to make its
characteristic independence explicit.

\begin{lemma}[Top-degree basis]\label{lem:W-basis}
If $n\ge k-1$, then the polynomials $w_{k,d}$, $|d|\le k-2$, form a basis
of $W_k$.  Hence
\[
 \dim W_k=M_{k-1}(n)=\dim V_k.
\]
\end{lemma}

\begin{proof}
They span by definition.  Divide by $x_1\cdots x_n$, and put
$m_d=2(k-|d|)-3$.  Suppose that
\[
 \sum_{|d|\le k-2}c_d
 \left(x_1^{m_d}+\cdots+x_n^{m_d}\right)x^{2d}=0.
\]
Choose $d$ with $c_d\ne0$ and $m_d$ as small as possible.  Since
$|d|\le k-2<n$, some $d_i$ is zero.  Suppose that a summand indexed by
$e$, with $c_e\ne0$, contributes the monomial
$x_i^{m_d}x^{2d}$.  Parity forces its odd exponent to occur in $x_i$; hence
$e_j=d_j$ for $j\ne i$ and $2e_i+m_e=m_d$.  Minimality gives
$m_e\ge m_d$, so $e_i=0$, $m_e=m_d$, and $e=d$.  The coefficient of this
monomial is therefore $c_d\ne0$, a contradiction.  Thus the spanning
family is independent over every field.  Its cardinality is
$M_{k-1}(n)$.
\end{proof}

Let
\[
 \varphi_k(P)=[P]_{n+2k-3}
\]
denote the homogeneous part of degree $n+2k-3$.  We initially regard
$\varphi_k$ as taking values in the full homogeneous component; the explicit
basis below will show that its image lies in $W_k$.

\section{The Catalan polynomial}\label{sec:catalan}

For $s\ge2$, define
\begin{equation}\label{eq:Ns}
 N_s(x)=(x-1)+\sum_{a=1}^{s-1}(-1)^aC_{a-1}\bigl(x(x-1)\bigr)^a.
\end{equation}

\begin{lemma}[Two-point interpolation]\label{lem:Ns}
In $\Z[x]$,
\[
 N_s(x)\equiv0\pmod{(x-1)^s},
 \qquad
 N_s(x)\equiv2x-1\pmod{x^s}.
\]
The quotient $N_s(x)/(x-1)^s$ has degree $s-2$ and leading coefficient
$(-1)^{s-1}C_{s-2}$.
\end{lemma}

\begin{proof}
Put $x=1+u$ and $w=-u(1+u)$.  If
$\mathcal C(w)=\sum_{j\ge0}C_jw^j$, then
$f(w)=w\mathcal C(w)$ is the unique formal series with zero constant term
satisfying $f=w+f^2$.  Since $-u=w+(-u)^2$, uniqueness gives
$w\mathcal C(w)=-u$.  Therefore
\[
 u+\sum_{a\ge1}(-1)^aC_{a-1}u^a(1+u)^a=0.
\]
Truncating after $a=s-1$ changes only terms divisible by $u^s$, proving the
first congruence.  Moreover $x(x-1)$ is unchanged by $x\mapsto1-x$, so
\[
 N_s(x)-N_s(1-x)=2x-1.
\]
The first congruence applied to $N_s(1-x)$ gives the second.  Finally, the
$a=s-1$ summand has degree $2s-2$ and leading coefficient
$(-1)^{s-1}C_{s-2}$.
\end{proof}

Write $u_i=x_i-1$.  For $q\ge0$ and $\varepsilon\in\{0,1\}$, put
\[
 b_{q,\varepsilon}(x)=\bigl(x(x-1)\bigr)^q(x-1)^\varepsilon.
\]
The polynomial $b_{q,\varepsilon}$ is $x^{2q+\varepsilon}$ plus terms of
smaller degree, and every nonnegative degree has a unique form
$2q+\varepsilon$.  Hence these polynomials form a unitriangular basis of
$\Z[x]$.  Their tensor products form a unitriangular basis of
$\Z[x_1,\ldots,x_n]$; we call them \emph{blocks}.  Distinct blocks have
distinct leading monomials.  Therefore the degree of a polynomial is the
largest degree of a block with nonzero coordinate.

For a monomial $m$, let $[m]P$ denote its coefficient in $P$.  The expansion
\[
 N_s(x_i)=u_i+\sum_{a=1}^{s-1}(-1)^aC_{a-1}y_i^a
\]
suggests truncating $\prod_iN_s(x_i)$ by total $y$-weight.
Define
\begin{equation}\label{eq:ghat}
\ghat_s=(-1)^{s-1}
\sum_{\substack{S\subseteq[n],\ a_i\ge1\ (i\in S)\\
                 \sum_{i\in S}a_i\le s-1}}
 (-1)^{\sum_{i\in S}a_i}
 \left(\prod_{i\in S}C_{a_i-1}y_i^{a_i}\right)
 \left(\prod_{j\notin S}u_j\right),
\end{equation}
where the term for $S=\varnothing$ is
$(-1)^{s-1}\prod_ju_j$.

\begin{theorem}[Catalan truncation]\label{thm:ghat}
For every $s\ge2$ and $n\ge1$:
\begin{enumerate}
\item $\ghat_s\in V_s^{\Z}$ and
      $\ghat_s(\zero)=(-1)^{n+s-1}$;
\item in $\Z[x_1,\ldots,x_n]$,
\begin{equation}\label{eq:ghat-top}
 [\ghat_s]_{n+2s-3}=C_{s-2}w_{s,0};
\end{equation}
\item if $\varnothing\ne S\subseteq[n]$, $a_i\ge1$, and
      $\sum_{i\in S}a_i=s-1$, then
\begin{equation}\label{eq:layer}
 \left[\prod_{i\in S}x_i^{2a_i}
       \prod_{j\notin S}x_j\right]\ghat_s
 =\prod_{i\in S}C_{a_i-1}.
\end{equation}
\end{enumerate}
\end{theorem}

\begin{proof}
A block in \eqref{eq:ghat}, indexed by $S$ and the $a_i$, has total
$y$-weight $A=\sum_{i\in S}a_i\le s-1$.  Every monomial in the block
satisfies
\[
 \sum_{i=1}^n\left\lfloor\frac{e_i}{2}\right\rfloor\le A.
\]
The block has degree $n+2A-|S|$.  If $S\ne\varnothing$, this is at most
$n+2s-3$; if $S=\varnothing$, its degree is $n$.  Thus every retained block
is $s$-reduced.  The product $\prod_iN_s(x_i)$ vanishes to
order at least $s$ at every point of $\pcube$ by Lemma~\ref{lem:Ns}.
Every block deleted in forming $\ghat_s$ has $y$-weight at least $s$ and
hence vanishes to order at least $s$ at every cube point.  This proves the
required punctured-cube vanishing.  At the origin, every term with
$S\ne\varnothing$ vanishes, and the empty term has value
$(-1)^{n+s-1}$.

A nonzero block has degree $n+2s-3$ only when $A=s-1$ and $|S|=1$.
Its leading monomial gives $C_{s-2}w_{s,0}$, proving part 2.

For part 3, let
$m=\prod_{i\in S}x_i^{2a_i}\prod_{j\notin S}x_j$.
Suppose a block indexed by $T$ and positive integers $b_i$ contributes to
$m$.  Every $i\in S$ belongs to $T$, and $b_i\ge a_i$.  If
$j\in T\setminus S$, then $b_j=1$.  Hence
\[
 \sum_{i\in T}b_i\ge s-1+|T\setminus S|.
\]
The truncation allows weight at most $s-1$, so $T=S$ and $b_i=a_i$.
Exactly one block contributes, and its sign is
$(-1)^{s-1+\sum a_i}=1$.
\end{proof}

The singleton case of \eqref{eq:layer} is the Catalan top coefficient.  The
case $|S|=r$ describes the layer of degree $n+2s-2-r$ and will provide the
pivots in Section~\ref{sec:degree}.

\section{An explicit basis and a diagonal top map}\label{sec:basis}

For $d\in\Z_{\ge0}^n$ with $|d|\le k-2$, set
\begin{equation}\label{eq:Gkd}
 G_{k,d}=y^d\ghat_{k-|d|},
 \qquad y^d=\prod_{i=1}^ny_i^{d_i}.
\end{equation}

\begin{theorem}[Integral Catalan basis]\label{thm:basis}
Let $k\ge2$.
\begin{enumerate}
\item The polynomials $G_{k,d}$, $|d|\le k-2$, form a $\Z$-basis of
      $V_k^{\Z}$.  Their reductions form a basis of $V_k$ over every field.
\item Their initial forms at the origin are
\begin{equation}\label{eq:initial-G}
 \operatorname{in}_{\zero}G_{k,d}=(-1)^{n+k-1}x^d.
\end{equation}
Consequently, if $P=\sum_dc_dG_{k,d}\ne0$ over a field, then
\begin{equation}\label{eq:origin-order}
 \ord_{\zero}P=\min\{|d|:c_d\ne0\}.
\end{equation}
\item If $n\ge k-1$, then
\begin{equation}\label{eq:diagonal}
 \varphi_k(G_{k,d})=C_{k-|d|-2}\,w_{k,d}.
\end{equation}
Thus $\varphi_k(V_k)\subseteq W_k$, and the matrix of
$\varphi_k:V_k\to W_k$ in the bases $\{G_{k,d}\}$ and $\{w_{k,d}\}$ is
diagonal.
\end{enumerate}
\end{theorem}

\begin{proof}
Put $t=|d|$ and $s=k-t$.  The polynomial $y^d$ has order $t$ at every cube
point, whereas $\ghat_s$ has order at least $s$ on $\pcube$.  Hence
$G_{k,d}$ has the required vanishing.  A block of $G_{k,d}$ indexed by
$S$ and parts of total weight $A$ has degree
\[
 n+2(t+A)-|S|.
\]
If $S\ne\varnothing$, this is at most $n+2k-3$; if $S=\varnothing$, it is
at most $n+2k-4$.  Every monomial in that block satisfies
\[
 \sum_i\left\lfloor\frac{e_i}{2}\right\rfloor\le t+A\le k-1.
\]
Thus $G_{k,d}\in V_k^{\Z}$.

The initial form of $y^d$ at the origin is $(-1)^tx^d$, and
$\ghat_s(\zero)=(-1)^{n+s-1}$, proving \eqref{eq:initial-G}.  Over a field,
these distinct initial monomials make the $G_{k,d}$ linearly independent.
Their number is
\[
 \#\{d:|d|\le k-2\}=M_{k-1}(n)=\dim V_k,
\]
so they form a basis, and the same argument gives
\eqref{eq:origin-order}.

It remains to check the integral assertion.  Order the multi-indices first
by $|d|$.  Under the origin-data isomorphism
\eqref{eq:integral-origin-data}, the columns belonging to $G_{k,d}$ have
zero entries in degrees below $|d|$, while their degree-$|d|$ block is
$(-1)^{n+k-1}$ times the identity.  The resulting integer matrix is
triangular with determinant $\pm1$.  Hence the $G_{k,d}$ form a $\Z$-basis
of $V_k^{\Z}$.

Finally, the top part of $y^d$ is $x^{2d}$.  Equation
\eqref{eq:ghat-top} gives
\[
 \varphi_k(G_{k,d})
 =x^{2d}C_{s-2}w_{s,0}
 =C_{k-|d|-2}w_{k,d},
\]
as claimed.
\end{proof}

For $n\ge k-1$, let
\[
 W_k^{\Z}=\bigoplus_{|d|\le k-2}\Z w_{k,d}.
\]
Theorem~\ref{thm:basis} shows that the integral top map
$\varphi_k^{\Z}:V_k^{\Z}\to W_k^{\Z}$ is well defined.

\begin{corollary}[Integral diagonal presentation]\label{cor:kernel}
Assume $n\ge k-1$.  Then
\begin{equation}\label{eq:cokernel}
 \operatorname{coker}\varphi_k^{\Z}\cong
 \bigoplus_{t=0}^{k-2}
 \left(\Z/C_{k-t-2}\Z\right)^{\binom{n+t-1}{n-1}}.
\end{equation}
Over a field $F$,
\begin{align}
 \ker\varphi_k
 &=\operatorname{span}\{G_{k,d}:C_{k-|d|-2}=0\text{ in }F\},
 \label{eq:kernel-basis}\\
 \dim\ker\varphi_k
 &=\sum_{\substack{0\le t\le k-2\\C_{k-t-2}=0\text{ in }F}}
     \binom{n+t-1}{n-1}.
 \label{eq:nullity}
\end{align}
With corresponding orderings of the two integral bases,
\begin{equation}\label{eq:determinant}
 \det\varphi_k^{\Z}=
 \prod_{j=0}^{k-2}C_j^{\binom{n+k-j-3}{n-1}}.
\end{equation}
\end{corollary}

\begin{proof}
In the $\Z$-basis of Theorem~\ref{thm:basis} and the defining basis
$\{w_{k,d}\}$ of $W_k^{\Z}$, the matrix is diagonal; the entry indexed by a
multi-index of weight $t$ is $C_{k-t-2}$.  There are
$\binom{n+t-1}{n-1}$ such indices.  This proves \eqref{eq:cokernel};
reduction to $F$ gives
\eqref{eq:kernel-basis} and \eqref{eq:nullity}.  Thus rank loss in
characteristic $p$ is precisely the $p$-torsion visible in the integral
cokernel.  Multiplying the diagonal entries gives \eqref{eq:determinant}.
\end{proof}

The integral presentation also records what happens over rings with torsion.
For a commutative ring $R$, tensoring Lemma~\ref{lem:interpolation}
identifies $V_k^{\Z}\otimes_{\Z}R$ with the module of $k$-reduced
polynomials over $R$ whose punctured-cube Hasse data of order less than $k$
vanish.

\begin{corollary}[Base change]\label{cor:base-change}
Let $R$ be a commutative ring, and put
\[
 V_{k,R}=V_k^{\Z}\otimes_{\Z}R,
 \qquad
 W_{k,R}=W_k^{\Z}\otimes_{\Z}R.
\]
If $n\ge k-1$, then the map
$\varphi_{k,R}=\varphi_k^{\Z}\otimes_{\Z}R$ is diagonal in the bases induced
by $\{G_{k,d}\}$ and $\{w_{k,d}\}$.  In particular,
\begin{align}
 \ker\varphi_{k,R}
 &=\bigoplus_{|d|\le k-2}
   \operatorname{Ann}_R(C_{k-|d|-2})G_{k,d},
 \label{eq:ring-kernel}\\
 \operatorname{coker}\varphi_{k,R}
 &\cong\bigoplus_{t=0}^{k-2}
   \left(R/C_{k-t-2}R\right)^{\binom{n+t-1}{n-1}}.
 \label{eq:ring-cokernel}
\end{align}
Here $\operatorname{Ann}_R(c)=\{r\in R:cr=0\}$.
\end{corollary}

\begin{proof}
Tensor the diagonal matrix in Theorem~\ref{thm:basis} with $R$ and take the
kernel and cokernel coordinate by coordinate.
\end{proof}

For example, if $p$ is prime, $q\ge1$, and $R=\Z/p^q\Z$, then
\begin{equation}\label{eq:finite-ring-kernel}
 |\ker\varphi_{k,R}|=
 p^{\displaystyle
 \sum_{t=0}^{k-2}\binom{n+t-1}{n-1}
 \min\{q,v_p(C_{k-t-2})\}}.
\end{equation}
Thus the integral presentation records not only rank loss modulo $p$, but
also its depth modulo every power of $p$.

If $k_0$ is the least $k$ for which $C_{k-2}=0$ in $F$, then
\eqref{eq:nullity} gives $\dim\ker\varphi_{k_0}=1$.  In general, the
weight-$t$ summand belongs to the kernel exactly when
$C_{k-t-2}=0$ in $F$.

\section{A degree-compatible basis}\label{sec:degree}

The diagonal map detects the top layer.  To control all lower layers, we use
the block coordinates introduced in Section~\ref{sec:catalan}.  We first
record what they say in arbitrary dimension.

For $s\ge1$, define
\begin{equation}\label{eq:kappa-def}
 \kappa_{F,n}(s)=
 \min\left\{r+2(s-1-A):
 \begin{array}{c}
 0\le r\le n,\quad A=a_1+\cdots+a_r\le s-1,\\[2pt]
 a_i\ge1,\quad C_{a_i-1}\ne0\text{ in }F
 \end{array}\right\}.
\end{equation}
The empty tuple is permitted, so $\kappa_{F,n}(1)=0$.

\begin{proposition}[Individual degrees]\label{prop:individual-degree}
For every field $F$, every $n\ge1$, and every $|d|\le k-2$,
\begin{equation}\label{eq:degree-G-general}
 \deg G_{k,d}=n+2k-2-\kappa_{F,n}(k-|d|).
\end{equation}
Moreover,
\begin{equation}\label{eq:kappa-rho}
 \kappa_{F,n}(s)=\rho_F(s)\qquad\text{whenever }n\ge\rho_F(s).
\end{equation}
\end{proposition}

\begin{proof}
Put $t=|d|$ and $s=k-t$.  A block of $G_{k,d}$ indexed by a set $S$ of
size $r$ and admissible parts of total weight $A\le s-1$ has degree
\begin{equation}\label{eq:block-degree}
 2t+n+2A-r
 =n+2k-2-\bigl(r+2(s-1-A)\bigr).
\end{equation}
Conversely, every tuple in \eqref{eq:kappa-def} occurs by assigning its
parts to $r$ coordinates.  Since distinct blocks have distinct leading
monomials, the largest degree of a nonzero block is the degree of
$G_{k,d}$.  This proves \eqref{eq:degree-G-general}.

Given a tuple of weight $A$, append $s-1-A$ parts equal to $1$.  The result
is an admissible decomposition of $s-1$ with $r+s-1-A$ parts.  Therefore
\[
 \rho_F(s)\le r+s-1-A\le r+2(s-1-A),
\]
so $\rho_F(s)\le\kappa_{F,n}(s)$.  If $n\ge\rho_F(s)$, an optimal
decomposition in \eqref{eq:rho-def} is allowed in \eqref{eq:kappa-def}; it
has $A=s-1$ and cost $\rho_F(s)$.  This proves
\eqref{eq:kappa-rho}.
\end{proof}

\begin{theorem}[Degree echelon]\label{thm:degree-echelon}
Let $F$ be a field, let $k\ge2$, and let $n\ge k-1$.  For every nonzero
expansion
\[
 P=\sum_{|d|\le k-2}c_dG_{k,d}
\]
one has
\begin{equation}\label{eq:degree-echelon}
 \deg P=
 \max_{c_d\ne0}\bigl(n+2k-2-\rho_F(k-|d|)\bigr).
\end{equation}
\end{theorem}

\begin{proof}
For $s=k-|d|$ we have $\rho_F(s)\le s-1\le n$, so
Proposition~\ref{prop:individual-degree} gives
\begin{equation}\label{eq:degree-G}
 \deg G_{k,d}=n+2k-2-\rho_F(k-|d|).
\end{equation}
We must rule out cancellation among the leading blocks of distinct basis
elements.

Among the $d$ with $c_d\ne0$, choose one for which
$r=\rho_F(k-|d|)$ is smallest; among those, choose $d$ with $|d|$ largest.
Put $s=k-|d|$.  Since
\[
 n-|\operatorname{supp}d|\ge n-|d|\ge s-1\ge r,
\]
we may choose distinct coordinates
$I=\{i_1,\ldots,i_r\}$ on which $d$ vanishes and an optimal decomposition
$s-1=a_1+\cdots+a_r$.  Consider the pivot block
\begin{equation}\label{eq:pivot-block}
 \mathcal B_d=
 y^d\left(\prod_{\nu=1}^r y_{i_\nu}^{a_\nu}\right)
 \left(\prod_{j\notin I}u_j\right).
\end{equation}
Its coordinate in $G_{k,d}$ is the nonzero product
$\prod_\nu C_{a_\nu-1}$.

Suppose that $G_{k,e}$ also contributes to $\mathcal B_d$.  The block's
$u$-pattern first forces the subset used by $G_{k,e}$ to be exactly $I$.
It then forces $e_j=d_j$ for $j\notin I$.  Since $d_{i_\nu}=0$, the part
used at $i_\nu$ must be $a_\nu-e_{i_\nu}\ge1$.  Thus $|e|\ge|d|$, with
equality only when $e=d$.  These new parts sum to
\[
 s-1-(|e|-|d|)=k-|e|-1
\]
and have nonzero Catalan numbers.  Hence $\rho_F(k-|e|)\le r$.  If
$c_e\ne0$, the minimal choice of $r$ makes this an equality, and the maximal
choice of $|d|$ then forces $|e|=|d|$, hence $e=d$.

No other nonzero summand contributes to $\mathcal B_d$.  The pivot survives
and has degree $n+2k-2-r$.  Every nonzero summand has degree at most this
value by \eqref{eq:degree-G}, so the surviving degree is exactly the maximum
in \eqref{eq:degree-echelon}.
\end{proof}

The proof exhibits a block pivot for every basis element.  It is stronger
than linear independence: no linear combination can cancel the highest
possible layer of all its nonzero summands.  Thus the basis also describes
the degree filtration explicitly.

\begin{corollary}[Origin-order and degree filtrations]\label{cor:bifiltration}
For integers $0\le q\le k-1$ and $D$, let
\[
 V_k^{\ge q,\le D}
 =\{P\in V_k:\ord_{\zero}P\ge q,\ \deg P\le D\}.
\]
If $n\ge k-1$, then $V_k^{\ge q,\le D}$ has basis
\[
 \{G_{k,d}: |d|\ge q,\ 
       n+2k-2-\rho_F(k-|d|)\le D\}.
\]
Consequently,
\begin{equation}\label{eq:bifiltration-dimension}
 \dim V_k^{\ge q,\le D}
 =\sum_{\substack{q\le t\le k-2\\
          n+2k-2-\rho_F(k-t)\le D}}
   \binom{n+t-1}{n-1}.
\end{equation}
\end{corollary}

\begin{proof}
Equation \eqref{eq:origin-order} selects the indices with $|d|\ge q$,
and Theorem~\ref{thm:degree-echelon} selects those whose degree is at most
$D$.
\end{proof}

\section{The exact minimum degree}\label{sec:extremal}

\begin{proof}[Proof of Theorem~\ref{thm:extremal}]
First suppose that $0\le\ell\le k-2$.  Choose any
$d\in\Z_{\ge0}^n$ with $|d|=\ell$.  By Theorems~\ref{thm:basis}
and~\ref{thm:degree-echelon}, the polynomial $G_{k,d}$ has the required
vanishing and degree
\[
 n+2k-2-\rho_F(k-\ell).
\]
This proves the upper bound.

For the lower bound, let $Q$ satisfy the prescribed vanishing conditions.
Lemma~\ref{lem:reduction} gives $P\in V_k$ with $\deg P\le\deg Q$ and
$\ord_{\zero}(Q-P)\ge k-1$.  Since $\ell\le k-2$, we have
$\ord_{\zero}P=\ell$.  Expand $P=\sum_dc_dG_{k,d}$.  By
\eqref{eq:origin-order}, the least $|d|$ with $c_d\ne0$ is $\ell$.
Theorem~\ref{thm:degree-echelon} gives
\[
 \deg Q\ge\deg P\ge n+2k-2-\rho_F(k-\ell).
\]

Now let $\ell=k-1$.  The polynomial
\begin{equation}\label{eq:boundary-example}
 H=x_1^{k-1}(x_1-1)^k\prod_{i=2}^n(x_i-1)
\end{equation}
has degree $n+2k-2$, order $k-1$ at the origin, and order at least $k$ at
every point of $\pcube$.  This proves the upper bound for every $n\ge1$.

For the lower bound, suppose instead that a polynomial $Q$ with the required
vanishing has degree at most $n+2k-3$.  The stronger part of
Lemma~\ref{lem:reduction} gives $P\in U_k$ such that $Q-P$ has order at
least $k$ at every cube point.  Hence $P$ has order at least $k$ on
$\pcube$ and order exactly $k-1$ at the origin.  In particular, $P\ne0$.
But every coordinate of $P$ recorded by the injective map in
Lemma~\ref{lem:interpolation} is zero, a contradiction.  Therefore
$\deg Q\ge n+2k-2$.  Since $\rho_F(1)=0$, this is
\eqref{eq:main-formula}.
\end{proof}

\section{Catalan arithmetic}\label{sec:arithmetic}

We first treat odd primes.  The next digit criterion is a direct consequence
of Kummer's carry theorem; it is also contained in the general divisibility
results of Alter and Kubota \cite{AlterKubota}.

\begin{lemma}[Odd-prime criterion]\label{lem:odd-criterion}
Let $p=2h+1$ be an odd prime, and write a positive integer $a$ in base $p$ as
\[
 a=a_0+a_1p+\cdots+a_mp^m,
 \qquad 0\le a_i<p.
\]
Then
\begin{equation}\label{eq:odd-digit-criterion}
 p\nmid C_{a-1}
 \quad\Longleftrightarrow\quad
 a_0\le h+1\ \text{ and }\ a_i\le h\ (i\ge1).
\end{equation}
\end{lemma}

\begin{proof}
Put $r=v_p(a)$.  Kummer's theorem identifies
$v_p\binom{2a-2}{a-1}$ with the number of carries in the base-$p$ addition
$(a-1)+(a-1)$, while
\[
 v_p(C_{a-1})=v_p\binom{2a-2}{a-1}-r.
\]
If $r=0$, the units digit of $a-1$ is $a_0-1$ and the higher digits are
$a_i$.  There are no carries precisely when
$2(a_0-1)<p$ and $2a_i<p$ for $i\ge1$, which is
\eqref{eq:odd-digit-criterion}.

If $r>0$, the first $r$ digits of $a-1$ are $p-1$; they force exactly $r$
carries.  The next digit is $a_r-1$ and receives an incoming carry.  No
additional carry occurs precisely when
$2(a_r-1)+1<p$ and $2a_i<p$ for $i>r$.  Since
$a_0=\cdots=a_{r-1}=0$, this is again
\eqref{eq:odd-digit-criterion}.
\end{proof}

\begin{proposition}[Odd-prime value of $\rho$]\label{prop:rho-odd}
For every odd prime $p$, one has $\rho_p(1)=0$, and for every $s\ge2$,
\[
 \rho_p(s)=
 \begin{cases}
  1,&p\nmid C_{s-2},\\
  2,&p\mid C_{s-2}.
 \end{cases}
\]
\end{proposition}

\begin{proof}
Put $N=s-1$ and $p=2h+1$.  If $C_{N-1}$ is nonzero modulo $p$, then
$N$ itself is an admissible part.  Otherwise, split every base-$p$ digit of
$N$ as a sum of two digits: for positions $i\ge1$, use two digits at most
$h$; in the units position, use two digits at most $h+1$.  No carries are
needed because every digit of $N$ is at most $2h$.  Lemma~\ref{lem:odd-criterion}
shows that the two resulting integers are admissible.  They can be chosen
positive: since $N$ is not admissible, either a higher digit exceeds $h$ or
the units digit exceeds $h+1$, and that offending digit may be split into
two positive summands.  Hence $\rho_p(s)=2$.
\end{proof}

Combining Theorem~\ref{thm:extremal} with
Proposition~\ref{prop:rho-odd} gives Corollary~\ref{cor:odd-intro}.
In particular, odd characteristic never lowers a minimum with
$\ell\le k-2$ by more than one degree.

For $p=2$, let $s_2(m)$ be binary digit sum.

\begin{proposition}[Binary value of $\rho$]\label{prop:rho-two}
For every $s\ge1$,
\[
 \rho_2(s)=s_2(s-1).
\]
\end{proposition}

\begin{proof}
For $s=1$, both sides are zero.  For $s\ge2$, Legendre's formula gives
\[
 v_2(C_m)
 =v_2((2m)!)-2v_2(m!)-v_2(m+1)
 =s_2(m+1)-1.
\]
Thus $C_{a-1}$ is odd precisely when $a$ is a power of $2$.
The binary expansion expresses $s-1$ as a sum of $s_2(s-1)$ powers of $2$.
Conversely, a sum of $r$ powers of $2$ has binary digit sum at most $r$,
because carrying can only decrease digit sum.
\end{proof}

Proposition~\ref{prop:rho-two} and Theorem~\ref{thm:extremal} give
Corollary~\ref{cor:two-intro}.  Unlike the odd-characteristic case,
the loss can exceed one degree.  For $0\le\ell\le k-2$, the drop from the
characteristic-zero value $n+2k-3$ is
\[
 s_2(k-\ell-1)-1.
\]

\section{Remarks and further directions}\label{sec:remarks}

\paragraph{Smaller dimension.}
Proposition~\ref{prop:individual-degree} determines the degree of every
individual basis element for all $n$, but below the stable range their
leading blocks can cancel.  The univariate problem already shows that the
hypothesis $n\ge k-1$ is substantive: for every field,
\begin{equation}\label{eq:univariate}
 \delta_F(1;k,\ell)=k+\ell,
\end{equation}
because a polynomial with order at least $k$ at $1$ is divisible by
$(x-1)^k$, its remaining factor must have order $\ell$ at $0$, and
$x^\ell(x-1)^k$ attains equality.
For example, over $\mathbb F_2$,
\[
 G_{3,0}+G_{3,(1)}=(x-1)^3.
\]
Both summands have degree $4$, but their sum has degree $3$.  Thus a
small-dimensional formula must account for cancellations between the
Catalan basis elements, not merely their individual degrees.

\paragraph{Larger grids.}
Huang, Wang, and Wang extend the characteristic-zero theorem from the
Boolean cube to $\{0,1,\ldots,m\}^n$ \cite{HuangWangWang}.  Their proof uses
Lagrange inversion and a more general top-degree space.  It would be
interesting to identify an analogue of the Catalan basis $G_{k,d}$ there.
Such a basis might reveal the relevant modular coefficients and determine
positive-characteristic drops for $mB^n$.

\paragraph{Products of linear forms.}
The polynomial minimum gives a lower bound for an almost $k$-cover by
hyperplanes, but the extremal polynomials constructed here need not split
into linear factors.  The present argument therefore does not determine the
corresponding geometric minimum in positive characteristic; in particular,
it does not decide whether the first polynomial drop can be realized by a
product of linear forms.

\end{document}